\documentclass[a4paper,11pt]{article}

\usepackage{hyperref}
\usepackage{stmaryrd}
\usepackage{graphicx,xcolor}
\usepackage{amsfonts}
\usepackage{amsmath,amssymb,mathrsfs,amsthm}
\usepackage[all]{xy} 
\usepackage{color}
\usepackage{times}
\usepackage{enumerate,vmargin}
\usepackage{verbatim}
\usepackage{paralist}
\usepackage[labelfont=bf,size=small,font=it]{caption}
\usepackage[numbers,sort&compress]{natbib}

\setpapersize{A4}
\setmarginsrb{2.5cm}{2cm}{2.5cm}{2cm}{.4cm}{4mm}{0cm}{7.5mm}

\hypersetup{
    unicode      = false,     % non-Latin characters in Acrobatâs bookmarks
    pdftoolbar   = true,      % show Acrobatâs toolbar?
    pdfmenubar   = true,      % show Acrobatâs menu?
    pdffitwindow = true,      % page fit to window when opened
    pdfnewwindow = true,      % links in new window
    colorlinks   = true,      % false: boxed links; true: colored links
    linkcolor    = blue,      % color of internal links
    citecolor    = blue,      % color of links to bibliography
    filecolor    = blue,      % color of file links
    urlcolor     = blue       % color of external links
}

\theoremstyle{plain}
\newtheorem{lem}{Lemma}
\newtheorem{thm}[lem]{Theorem}
\newtheorem{prop}[lem]{Proposition}

\theoremstyle{definition}

\theoremstyle{remark}

\newcommand{\N}{\mathbb N}
\newcommand{\R}{\mathbb R}

\newcommand{\cT}{\mathcal T}
\newcommand{\cP}{\mathcal P}
\newcommand{\cE}{\mathcal E}
\newcommand{\cR}{\mathcal R}
\newcommand{\cN}{\mathcal N}
\newcommand{\tcT}{\tilde{\cT}}
\newcommand{\e}{\mathrm{e}}
\newcommand{\Ht}{\mathrm{ht}}
\newcommand{\llb}{\llbracket}
\newcommand{\rrb}{\rrbracket}

\newcommand{\bcT}{\cT}

\title{\Large\textsc{$k$-cut model for the Brownian Continuum Random Tree}}

\author{Minmin Wang
\thanks{Department of Mathematics, University of Sussex, Falmer, Brighton BN1 9QH, United Kingdom. Email: minmin.wang@sussex.ac.uk}}

\date{\today}

\begin{document}

\maketitle
\begin{abstract}
To model the destruction of a resilient network, Cai, Holmgren, Devroye and  Skerman \cite{CHDS19} introduced the $k$-cut model on a random tree, as an extension to the classic problem of cutting down random trees. 
Berzunza, Cai and Holmgren \cite{BCH19} later proved that the total number of cuts in the $k$-cut model to isolate the root of a Galton--Watson tree with a finite-variance offspring law and conditioned to have $n$ nodes, when divided by $n^{1-1/2k}$, converges in distribution to some random variable defined on the Brownian CRT. We provide here a direct construction of the limit random variable, relying upon the Aldous--Pitman fragmentation process and a deterministic time change. %We also study some distributional properties of this random variable. 
\end{abstract}

\section{Introduction}
Let $k\in \N$ and let $T$ be a rooted tree. The following procedure is considered by Cai, Holmgren, Devroye and  Skerman \cite{CHDS19}. To each vertex $v$ of $T$, we associate an independent Poisson process $N_{v}=(N_{v}(t))_{t\ge 0}$ of rate $1$. Imagine that each time $N_{v}$ increases, the vertex $v$ is cut once and is eventually removed when it receives $k$ cuts.  The procedure ends when the root is removed. We are interested in the total number of cuts, denoted as $X_{k}(T)$. Let us observe that for $k=1$, the above procedure reduces to the classic problem of cutting down random trees introduced by Meir and Moon \cite{MeMo70}; see in particular \cite{Ja06, ABH14, BM13, AD13, BrWa17a, Di15} for some recent progress on the classical version. 

Let $\xi=(\xi(p))_{p\ge 0}$ be a probability measure on the set of non negative integers which satisfies
\[
\sum_{p\ge 1}p \,\xi(p)=1, \quad \text{ and } \quad 0<\sigma:=\Big(\sum_{p\ge 2}p(p-1)\xi(p)\Big)^{1/2}<\infty. 
\]
For $n\ge 1$, let $T_{n}$ be a Galton--Watson tree with offspring distribution $\xi$ conditioned on having $n$ vertices. Berzunza, Cai and Holmgren show in \cite{BCH19} that 
\begin{equation}
\label{cv: k-cut}
\Big(\frac{\sigma}{\sqrt{n}}T_{n}, \frac{X_{k}(T_{n})}{\sigma^{\frac{1}{k}}n^{1-\frac{1}{2k}}}\Big) \xrightarrow[n\to\infty]{d} (\cT, Z_{k}),
\end{equation}
where $\cT$ is the so-called Brownian Continuum Random Tree, and $Z_{k}$ is a non degenerate random variable whose distribution is characterised via its moments. 
Note that the convergence of $\frac{\sigma}{\sqrt{n}}T_{n}$ to $\cT$, due to Aldous \cite{aldcrt3}, is well known and takes place in the weak topology of the Gromov--Hausdorff space. We defer the formal definitions of these objects till a later point. 
Let us also point out that the joint convergence in \eqref{cv: k-cut} generalises an earlier result for $k=1$ by Janson \cite{Ja06}. 

In the case $k=1$, it is also known that $Z_{1}$ can be explicitly written as a functional of the so-called Aldous--Pitman fragmentation process, thanks to the works of Addario-Berry, Broutin \& Holmgren \cite{ABH14}, Bertoin \& Miermont \cite{BM13}, Abraham \& Delmas \cite{AD13}.  In this work, we extend this construction of $Z_{1}$ to the general setting of $k\ge 1$, thus answering a question in \cite{BCH19} on the construction of $Z_{k}$. 
 To that end, let us start with a brief introduction to the Aldous--Pitman fragmentation process. 

The Aldous--Pitman fragmentation process can be viewed as the analogue of the $1$-cut procedure for the Brownian continuum random tree (CRT). 
First, we need to construct this CRT. Let us take $\e=(e_{s})_{0\le s\le 1}$, where $\frac12 \e$ is distributed as the standard normalised Brownian excursion of duration $1$. For $s, t\in [0, 1]$, define
\[
d(s, t)=e_{s}+e_{t}-2 b(s, t), \quad \text{ where }\quad b(s,t)=\min_{s\wedge t \le u \le s\vee t}e_{u}.
\]
It turns out the function $d$ is non negative, symmetric and satisfies the triangular inequality. To turn it into a metric, let $s\sim t$ if and only if $d(s, t)=0$. Then $d$ defines a metric on the quotient space $\cT:=[0, 1]/\!\!\sim$, which we still denote as $d$. 
In the sequel, we will refer to the (random) metric space $(\cT, d)$ as the Brownian CRT. 
Note that it has ``tree-like'' features: each pair of points in $\cT$, say $\sigma$ and $\sigma'$, is joined by a unique path,  denoted as $\llb \sigma, \sigma'\rrb$, 
which turns out to be a geodesic. Metric spaces with such properties are called $\R$-trees. Interested readers can check Evans \cite{Evans} and Le Gall \cite{LeG05}  for more background on $\R$-trees and CRT. 

Let us also introduce the following notation on $(\cT, d)$ which will be useful later. 
We denote by $p: [0, 1]\to \cT$ the canonical projection which sends every $t\in [0, 1]$ to its equivalence class with respect to $\sim$. The {\it root} of $(\cT, d)$ is then the point $\rho=p(0)=p(1)$. 
In addition, the map $p$ also induces a probability measure on $\cT$: the {\it mass measure}, denoted as $\mu$, is the push-forward of the uniform measure on $[0, 1]$ by $p$. 
On the other hand, the {\it length measure} $\ell$ is a $\sigma$-finite measure on $\cT$, characterised by the relation $\ell(\llb \sigma, \sigma'\rrb)=d(\sigma, \sigma')$, for all $\sigma, \sigma'\in \cT$. 

We introduce a Poisson point measure $\cP(dt, dx)=\sum_{i\ge 1}\delta_{(t_{i}, x_{i})}(dt, dx)$ on $\R_{+}\times \cT$ of intensity $dt\, \ell(dx)$. One can imagine the $(t_{i}, x_{i})$'s as cuts on $\cT$: at time $t_{i}$, the point $x_{i}$ is removed from $\cT$, which disconnects the tree. As times moves on, more cuts arrive and $\cT$ fragments into finer and finer connected components. 
The Aldous--Pitman fragmentation consists in describing the time evolution of the collection of $\mu$-masses of these connected components. It is also known that the above cutting process of $\cT$ using points from $\cP$ appears as the scaling limit of the $1$-cut procedure on $T_{n}$. On the other hand, the key element in our construction is the following {\it time-changed} version of $\cP$: for $k\in [1, \infty)$, define
\begin{equation}
\label{defcP}
\tilde\cP=\sum_{i\ge 1}\delta_{(s_{i}, \,x_{i})}, \quad \text{where} \quad s_{i}=\big(\Gamma(k+1)t_{i}\big)^{\frac{1}{k}},\  i\ge 1.
\end{equation}
Here, $\Gamma(\cdot)$ stands for the Gamma function. 
%In particular, $\mathcal N_{\sigma}(t)=\cP([0, t]\times \llb \rho, \sigma\rrb)$, $t\ge 0$, has the distribution of a Poisson process of rate $\ell(\llb \rho, \sigma\rrb)$. 
Let us denote by $\cT_{t}=\{\sigma\in \cT: \cP([0, t]\times \llb \rho, \sigma\rrb)=0\}$, the subtree connected to the root at time $t$. 
Similarly, denote $\tilde\cT_{t}=\{\sigma \in \cT: \tilde\cP([0, t]\times \llb \rho, \sigma\rrb)=0\}$ the remaining subtree in the time-changed cutting process. 
We define
\begin{equation}
\label{def: Xk}
X_{k}(\cT)=\int_{0}^{\infty}\mu(\tilde\cT_{t})dt =\int_{0}^{\infty} \mu\big(\cT_{t^{k}/\Gamma(k+1)}\big)dt=\frac{(\Gamma(k+1))^{\frac{1}{k}}}{k}\int_{0}^{\infty}\mu(\cT_{s})\, s^{\frac{1}{k}-1}ds. 
\end{equation}
For $k=1$, $X_{1}(\cT)$ appears in \cite{BM13, AD13, ABH14} as the scaling limit of $X_{1}(T_{n})$. Let us also recall Aldous and Pitman \cite{AP98} have shown that the process $(\mu(\cT_{t}))_{t\ge 0}$ has the same distribution as $((1+L_{t})^{-1})_{t\ge 0}$ with $(L_{t})_{t\ge 0}$ being a $\frac12$-stable subordinator. Combined with a Lamperti time-change, this then implies $X_{1}(\cT)$ has the Rayleigh distribution (\cite{BM13}). Note that we also have the following bound from \eqref{def: Xk}.
\begin{equation}
\label{bd-X}
k\,\Gamma(k+1))^{-\frac1k}  X_{k}(\bcT)\le \int_{0}^{1}s^{\frac{1}{k}-1}ds +\int_{1}^{\infty}\mu(\cT_{s})ds\le k+X_{1}(\bcT).
\end{equation}
So in particular, $X_{k}(\cT)<\infty$, a.s.
Let us also point out that even though the discrete model is only defined for $k\in \N$, the above definition of $X_{k}(\cT)$ makes sense for all $k\in [1, \infty)$. Here is our main result.  

\begin{thm}
\label{thm1}
For all $k\in \N$, conditional on $(\cT, d)$, $X_{k}(\cT)$ has the same distribution as $Z_{k}$. 
\end{thm}

We'll give two proofs to the theorem. In Section \ref{sec2}, we give a first proof by identifying the conditional moments of $X_{k}(\cT)$ given $\cT$ with those of $Z_{k}$, which were computed  in \cite{BCH19}. In Section \ref{sec3}, we give a second proof via weak convergence arguments. Even it takes a bit more space, the second proof is perhaps more helpful in explaining the motivation for the definition \eqref{def: Xk}, as well as provides an alternative proof to the convergence in \eqref{cv: k-cut}.

\section{Conditional expectation of $X_{k}(\bcT)$ given $\bcT$}
\label{sec2}

We will need the following notation. For $q\in \N$ and $\mathbf s=(s_{1}, s_{2}, \dots, s_{q})\in [0, 1]^{q}$, we set 
$\Delta^{\e}_{1}(\mathbf s)=e_{s_{1}}$, and more generally for $2\le r\le q$, 
\[
\Delta^{\e}_{r}(\mathbf s)=e_{s_{r}}-\max_{i<r}b(s_i, s_{r}), \quad \text{where }\ b(s,t)=\min_{s\wedge t \le u \le s\vee t}e_{u}.
\]
%where $(s_{(1)}, s_{(2)}, \dots, s_{(q)})$ is the re-ordering of $(s_{i})_{1\le i\le q}$ in non decreasing order. 
Note that $\Delta^{\e}_{1}(\mathbf s)+\cdots+\Delta^{\e}_{r}(\mathbf s)$ is the total length (i.e.~$\ell$-mass) of the reduced subtree of $\cT$ spanned by $p(s_{1}), \dots, p(s_{r})$, for all $r\le q$. Our goal is to prove the following formulas on the moments of $X_{k}(\cT)$. 

\begin{prop}
\label{prop: moments}
For all $k\ge 1$ and $q\in \N$, we have
\begin{align}
\notag
\mathbb E[X_{k}(\cT)^{q}\,|\,\e\,]=&\,q!\int_{[0, 1]^{q}}ds_{1}\cdots ds_{q} \int_0^{\infty}\int_{0}^{x_{1}}\cdots\int_{0}^{x_{q-1}} \\ &\exp\left(-\frac{1}{k!}\left(\Delta^{\e}_{1}(\mathbf s)x_{1}^{k}+\Delta^{\e}_{2}(\mathbf s)x_{2}^{k}+\cdots+\Delta^{\e}_{q}(\mathbf s)x_{q}^{k}\right)\right)dx_{q}\cdots dx_{1}. 
\label{moments}
\end{align}
\end{prop}

%\begin{rem}
%For $k=1$, we can further simplify \eqref{moments} into
%\[
%\mathbb E[X_{1}(\bcT)^{q}\,|\,e]=q!\int_{[0, 1]^{q}}ds_{1}\cdots ds_{q} \frac{1}{\Delta^{e}_{1}(\mathbf s)(\Delta^{e}_{1}(\mathbf s)+\Delta^{e}_{2}(\mathbf s))\cdots(\Delta^{e}_{1}(\mathbf s)+\cdots+\Delta^{e}_{q}(\mathbf s))}
%\]
%which first appeared in Janson \cite{Ja06}. 
%\end{rem}

\begin{proof}
%Recall the point measure $\tilde\cP$ from \eqref{defcP}. 
For $v\in \cT$, we define $\cE_{v}=\inf\{t>0: \tilde\cP([0, t]\times \llb \rho, v\rrb)=1\}$, the moment that $v$ is separated from the root. Then $v\in \tcT_{t}$ if and only if $\cE_{v}>t$. Therefore, we can re-write $X_{k}(\cT)^{q}$ as follows.
%We note that conditional on $\bcT$, $\cE_{\sigma}$ has the following distribution: for $t\ge 0$, 
%\[
%\mathbb P(\cE_{\sigma}>t \,|\, e) = \mathbb P(\mathcal N_{\sigma}((t^{k}/k!)=0) = \exp\left(-\frac{d(\rho, \sigma)t^{k}}{k!}\right), 
%\]
%since we have seen that $\mathcal N_{\sigma}$ is a Poisson process with rate $d(\rho, \sigma)$. 
%Recalling that $\mu$ is the push-forward of the Lebesgue measure on $[0, 1]$,  we have
\begin{align*}
\big(X_{k}(\cT)\big)^{q}&=\int_{\R_{+}^{q}}\mu(\tilde\cT_{t_{1}})\mu(\tilde\cT_{t_{2}})\cdots \mu(\tilde\cT_{t_{q}})dt_{1}dt_{2}\cdots dt_{q} \\
&= \int_{\R_{+}^{q}}\int_{\cT^{q}} \mathbf 1_{\{\mathcal E_{v_{1}}>t_{1}, \,\dots,\, \mathcal E_{v_{q}>t_{q}}\}}\mu(dv_{1})\cdots\mu(dv_{q})dt_{1}\cdots dt_{q}\\
&= \int_{\R_{+}^{q}}\int_{[0, 1]^{q}} \mathbf 1_{\{\mathcal E_{p(s_{1})}>t_{1},\, \dots,\, \mathcal E_{p(s_{q})}>t_{q}\}}ds_{1}\cdots ds_{q}\,dt_{1}\cdots dt_{q},
\end{align*}
where we have used in the last line the definition that $\mu$ is the push-forward of the Lebesgue measure on $[0, 1]$. 
Write $\mathbb E_{\e}$ as a shorthand for $\mathbb E[\cdot\,|\,\e\,]$. The above yields that 
\[
\mathbb E_{\e}[X_{k}(\bcT)^{q}]=\int_{[0, 1]^{q}} \int_{\R_{+}^{q}}\mathbb P_{\e}(\mathcal E_{p(s_{1})}>t_{1},\, \dots,\, \mathcal E_{p(s_{q})}>t_{q})dt_{1}\cdots dt_{q}\,ds_{1}\cdots ds_{q}.
\]
We then split $\R_{+}^{q}$ into $q!$ subdomains according to the $q!$ outcomes in ranking $(t_{i})_{1\le i\le q}$. However, $(s_{i})_{1\le i\le q}$ is sampled in an i.i.d fashion and is therefore exchangeable, so that integration from each subdomain will contribute equally. Hence,
\[
\mathbb E_{\e}[X_{k}(\bcT)^{q}]=q!\int_{[0, 1]^{q}} \int_{0}^{\infty}\int_{0}^{t_{1}}\dots\int_{0}^{t_{q-1}} \mathbb P_{\e}(\mathcal E_{p(s_{1})}>t_{1},\, \dots,\, \mathcal E_{p(s_{q})}>t_{q})dt_{q}\cdots dt_{1}\,ds_{1}\cdots ds_{q}. 
\]
\begin{figure}[tp]
\centering
\includegraphics[height=4.5cm]{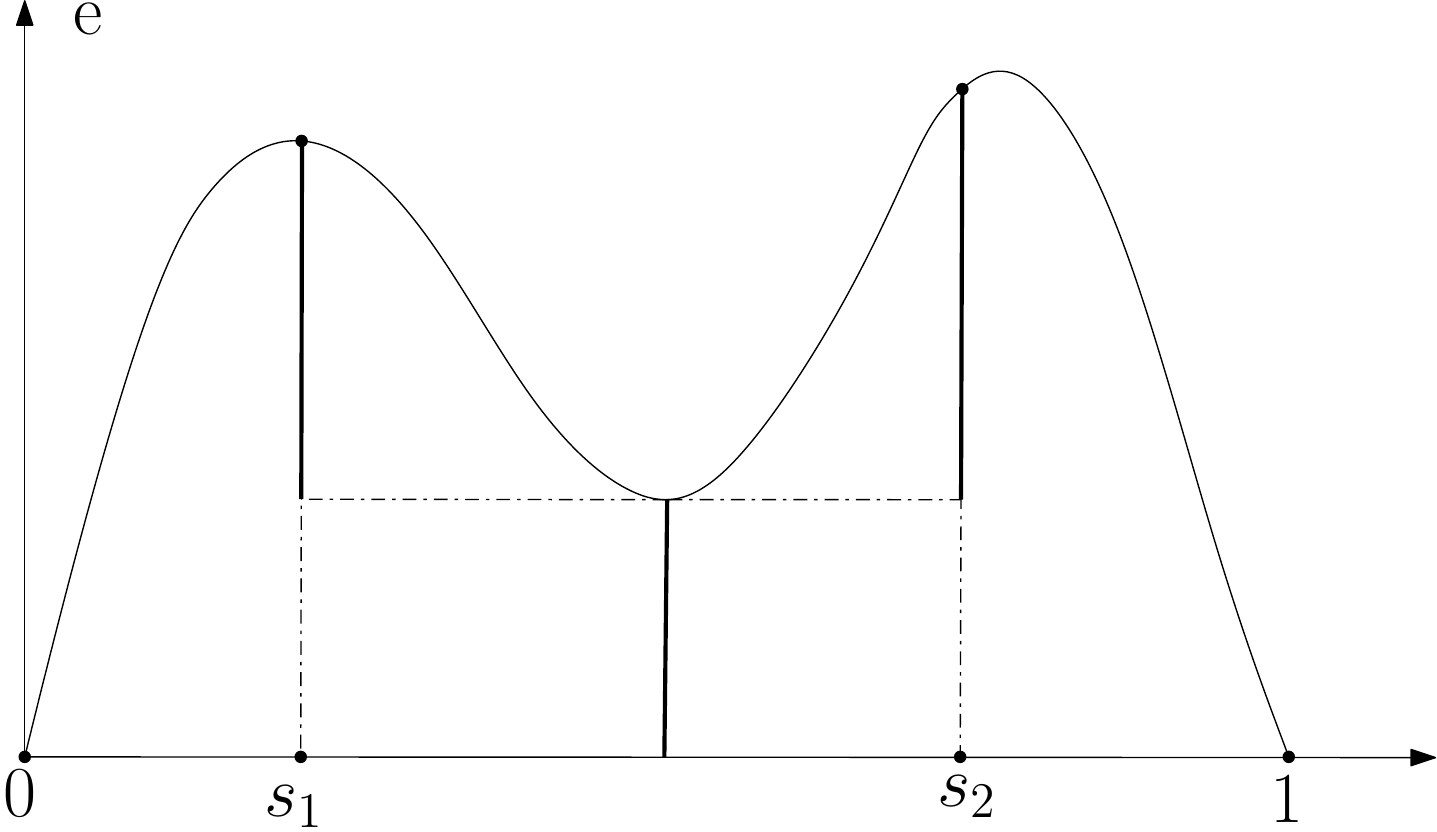}
\caption{\label{fig} An illustration of $\cR_{q}$ with $q=2$. Here, $\cR_{q}$ has the shape of a binary tree with $2$ leaves, one branch point and three edges (depicted by the line segments in bold). The edge lengths correspond to the lengths of these line segments.}
\end{figure}

Let $\cR_{q}$ be the reduced subtree of $\cT$ spanned by $v_{1}=p(s_{1}),\dots, v_{q}=p(s_{q})$, i.e.~the smallest connected subspace of $\cT$ containing these $q$ points and the root $\rho$. Note that $\cR_{q}$ is a ``finite'' tree in the sense that it only has a finite number of branch points and leaves. Here, it will be convenient to think of it as a (graph) tree $(\mathrm V_{q}, \mathrm E_{q})$, where the vertex set $\mathrm V_{q}$ consists of the root, the leaves and the branch points of $\cR_{q}$ and each edge $e\in \mathrm E_{q}$ is equipped with an edge length $l(e)\in (0, \infty)$. These edge lengths are consistent with the distance $d$ in the following way: for each $v\in \mathrm V_{q}$,  
$d(\rho, v)=\sum_{e\in P(v)} l(e)$, where $P(v)$ stands for the set of edges on the path from the root $\rho$ to $v$. See also Fig.~\ref{fig} for an example of $\cR_{q}$. Now to each edge $e$ in this tree, we associate an independent exponential variable $E_{e}$ of mean $1/l(e)$. It follows from the definition \eqref{defcP} of $\tilde\cP$ that $(\mathcal E_{v_{r}})^{k}/k!$ is distributed as an exponential random variable of mean $1/d(\rho, v_{r})=1/e_{s_{r}}$. It is then straightforward to check that
\[
\big(\mathcal E_{v_{r}}\,; \ 1\le r\le q\big) \overset{(d)}{=}\Big(\min_{e\in P(v_{r})}\, (k!E_{e})^{1/k}\,; \ 1\le r\le q\Big). 
\]
Bearing in mind that $t_{1}>t_{2}>\cdots>t_{q}$, we then find that 
\begin{align*}
\mathbb P_{\e}(\mathcal E_{v_{1}}>t_{1},\dots,\mathcal E_{v_{q}}>t_{q})&=\mathbb P_{\e}\big( (k!E_{e})^{\frac1k}>t_{r}, \forall\, e\in P(v_{r}), 1\le r\le q\big)\\
&=\mathbb P_{\e}\big( (k!E_{e})^{\frac1k}>t_{1}, \forall\, e\in P(v_{1})\big)\mathbb P_{\e}\big((k!E_{e})^{\frac1k}>t_{2}, \forall\, e\in P(v_{2})\!\setminus\! P(v_{1})\big) \\
 &\quad \cdots \mathbb P_{\e}\big((k!E_{e})^{\frac1k}>t_{2}, \forall\, e\in P(v_{q})\setminus \cup_{r<q}P(v_{r})\big)\\
&=\,\exp\left(-\frac{1}{k!}\left(\Delta^{\e}_{1}(\mathbf s)t_{1}^{k}+ \Delta^{\e}_{2}(\mathbf s)t_{2}^{k}+ \cdots \Delta^{\e}_{q}(\mathbf s) t_{q}^{k}\right)\right). 
\end{align*}
By the previous arguments, this completes the proof. 
\end{proof}

\begin{proof}[{\bf Proof 1 of Theorem \ref{thm1}}]
Comparing \eqref{moments} with equations (8) and (9) in \cite{BCH19}, we see that $\mathbb E[X_{k}(\cT)^{q}|\e]=\mathbb E[Z_{k}^{q}|\e]$ for all $q\in \N$. Applying Theorem 2 and Lemma  8 there, we conclude that conditional on $\e$, $X_{k}(\cT)$ has the same distribution as $Z_{k}$. 
\end{proof}

\section{Scaling limit of $X_{k}(T_{n})$}
\label{sec3}

Here, we give a second proof of the theorem by showing $X_{k}(\cT)$ is the scaling limit of $X_{k}(T_{n})$. 
Throughout this section, we assume $k\in\{2, 3, \dots\}$. 

\subsection{Convergence of random trees}

%{\it Gromov--Hausdorff, pointed GH, Aldous' Theorem. conv of subtrees, conv of the lengths. Skorohkold space. }
We briefly recall Aldous' Theorem on the convergence of the conditioned Galton--Watson tree $T_{n}$, as well as provide some necessary background on the Gromov--Hausdorff topology. Further details on these topics can be found in \cite{LeG05, Evans, Burago, Mi09, aldcrt3}. 

The {\it Gromov--Hausdorff distance} between two compact metric spaces $(X, d_{X})$ and $(Y, d_{Y})$ is the following quantity:
\[
d_{\mathrm{GH}}(X, Y)=\inf_{\phi, \varphi, Z}d_{Z, \mathrm{Haus}}(\phi(X), \varphi(Y)),
\]
where the infimum is over all the isometric embeddings $\phi: X\to Z$ and $\varphi: Y\to Z$ into a common metric space $(Z, d_{Z})$, and $d_{Z, \mathrm{Haus}}$ stands for the usual Hausdorff distance for the compact sets of $Z$. In our application, we often need to keep track of specified points in the initial spaces.  To that end, let $\mathbf x=(x_{1}, \dots, x_{p})$ and   $\mathbf y=(y_{1}, \dots, y_{p})$ be $p\in \N$ points of $X$ and $Y$. Then the {\it marked Gromov--Hausdorff distance} between $(X, d_{X}, \mathbf x)$ and $(Y, d_{Y}, \mathbf y)$ is defined as 
\[
d_{p, \mathrm{GH}}(X, Y)=\inf_{\phi, \varphi, Z}\Big(d_{Z, \mathrm{Haus}}(\phi(X), \varphi(Y))\vee \max_{1\le i\le p}d_{Z}\big(\phi(x_{i}), \varphi(y_{i})\big)\Big),
\]
where the infimum is again over all the isometric embeddings of $X$ and $Y$ into a common metric space. 
For each $p\ge 1$, it turns out that the space of metric spaces with $p$ marked points is a Polish space with respect to $d_{p, \mathrm{GH}}$ (\cite{Mi09}). 
Now the convergence of $T_{n}$ mentioned earlier can be given a precise meaning. Let us recall that the Brownian CRT $(\cT, d)$ is a metric space by definition. Recall also $\rho\in\cT$ stands for its root. Equipping its vertex set with the graph distance, we can also view the tree $T_{n}$ as a metric space. Let us denote by $\frac{\sigma}{\sqrt{n}}T_{n}$ the rescaled metric space where the graph distance is multiplied by a factor $\frac{\sigma}{\sqrt{n}}$. Denote also by $\rho_{n}$ its root. We have
\begin{equation}
\label{cv-tree}
\Big(\frac{\sigma}{\sqrt{n}}T_{n}, \rho_{n}\Big) \xrightarrow[n\to\infty]{(d)} \big(\cT, \rho),
\end{equation}
in the weak topology of the marked Gromov--Hausdorff distance. 

We note that $\cT$ is further equipped with a probability measure $\mu$. Let us define its discrete counterpart: for $n\ge 1$, let $\mu_{n}$ be the uniform probability measure on the vertex set of $T_{n}$. In fact, Aldous's Theorem in \cite{aldcrt3} also implies the following convergence of reduced trees. 
Given $\cT$, let $(V_{i})_{i\ge 1}$ be an i.i.d.~sequence of points in $\cT$  sampled with $\mu$. For $p\in \N$, denote by $\cR_{p}$ the reduced tree of $\cT$ spanned by $V_{1}, \dots, V_{p}$. Similarly, we sample an i.i.d.~sequence $(V^{n}_{i})_{i\ge 1}$ from $T_{n}$ with law $\mu_{n}$. Let $\cR^{n}_{p}$ be the reduced subtree of $T_{n}$ spanned by $V^{n}_{1}, \dots, V^{n}_{p}$, namely, the smallest subgraph of $T_{n}$ (an edge of the subgraph is also an edge of $T_{n}$) containing $V^{n}_{1}, \dots, V^{n}_{p}$ and the root $\rho_{n}$. As above, we denote by $\frac{\sigma}{\sqrt{n}}\cR^{n}_{p}$ the metric space obtained from $\cR^{n}_{p}$ by equipping its vertex set with $\frac{\sigma}{\sqrt{n}}$ times the graph distance. Then we have  
\begin{equation}
\label{cv-R}
\forall\,p\in\N, \quad 
\Big(\frac{\sigma}{\sqrt{n}}\cR^{n}_{p}, V^{n}_{1}, \dots, V^{n}_{p}\Big) \xrightarrow[n\to\infty]{(d)} (\cR_{p}, V_{1}, \dots, V_{p}),  
\end{equation}
with respect to the marked Gromov--Hausdorff topology. We have seen that $\cR_{p}$ can be viewed as a (graph) tree with edge lengths. But so does $\frac{\sigma}{\sqrt{n}}\cR^{n}_{p}$, where the edge length is simply $\frac{\sigma}{\sqrt{n}}$.  
In fact, the convergence in \eqref{cv-R} amounts to saying that the ``shape'' of $\cR^{n}_{p}$ coincides with that of $\cR_{p}$ for large $n$ and 
\begin{equation}
\label{cv-l}
\frac{\sigma}{\sqrt{n}}\,\#\cR^{n}_{p} \xrightarrow[n\to\infty]{(d)} \ell(\cR_{p}), \quad p=1, 2, \dots.
\end{equation}
where $\#$ stands for the counting measure on the vertex set of $\cR^{n}_{p}$ and $\ell$ is the length measure of $\cT$. 

Let us recall the Poisson point measure $\cP$ has an intensity $dt \,\ell(dx)$. Since $\ell(\cR_{p})<\infty$, there is a finite number of ``cuts''  $(t_{i}, x_{i})$ from $\cP$ which fall on $\cR_{p}$ before time $t$. So a convenient approach to studying the cutting of $\cT$ is first look at those cuts on $\cR_{p}$, $p\ge 1$. We'll also see the convergences in \eqref{cv-R} and \eqref{cv-l} will be our starting point for proving the convergence of $X_{k}(T_{n})$. 

\subsection{Convergence of the cutting process}

For each vertex $v$ of $T_{n}$, let us denote $\eta_{v}=\inf\{t: N_{v}(t)=k\}$, the time when $v$ is removed from $T_{n}$. We show here that the point measure $\mathcal P_{n}:=\sum_{v\in T_{n}}\delta_{(\eta_{v}, v)}$ converges in an appropriate sense to $\tilde\cP$. Let us start with the following observation. 

\begin{lem}
\label{lem: gam}
For each $m\in\N$, suppose $a_{m}\in (0, \infty)$ and let $(G_{m, i})_{1\le i\le m}$ be independent Gamma$(k, \frac{1}{a_{m}})$ random variables whose probability density function is given by $\frac{1}{(k-1)!}a_{m}^{k}x^{k-1}e^{-a_{m}x}$, $x>0$. Let 
\[
\cN_{m}(t)=\sum_{1\le i\le m}\mathbf 1_{\{G_{m, i}\le t\}}, \quad t\ge 0.
\]
If $m\,a_{m}^{k}\to a\in (0, \infty)$ as $m\to\infty$, then we have 
\[
\big(\cN_{m}(t)\big)_{t\ge 0} \xrightarrow[m\to\infty]{(d)} \big(\cN(t^{k}/k!)\big)_{t\ge 0} \quad \text{ in }\  \mathbb D(\R_{+}, \R), 
\]
where $(\cN(t))_{t\ge 0}$ is a Poisson process on $\R_{+}$ of rate $a$ and $\mathbb D(\R_{+}, \R)$ is the space of c\`adl\`ag functions endowed with the Skorokhod topology. 

\end{lem}

\begin{proof}
Let $G$ denote a Gamma$(k, 1)$ random variable and let $X$ be a Poisson random variable of mean $t$. We note that
\begin{equation}
\label{gam}
\mathbb P(G\le t)=\mathbb P(X\ge k)=\sum_{j=k}^{\infty}e^{-t}\frac{t^{j}}{j!}=\frac{t^{k}}{k!}+t^{k+1}\,R(t),
\end{equation}
where $R(\cdot)$ is bounded on any finite interval. Let $T>0$. For all $t\le T$ and $p\ge 0$, noting $\mathbb P(G_{m, 1}\le t)=\mathbb P(G\le a_{m}t)$, we deduce  that 
\begin{align*}
\mathbb P(\cN_{m}(t)=p)&=\binom{m}{p} \Big(\mathbb P(G_{m, 1}\le t)\Big)^{p} \Big(\mathbb P(G_{m, 1}>t)\Big)^{m-p}\\
&=\frac{\binom{m}{p}}{m^{p}}\bigg(\frac{m(a_{m}t)^{k}}{k!}+m(a_{m}t)^{k+1}R(a_{m}t))\bigg)^{p}\bigg(1-\frac{(a_{m}t)^{k}}{k!}+(a_{m}t)^{k+1}R(a_{m}t))\bigg)^{m-p}\\
&\to \frac{1}{p!}\Big(\frac{a\,t^{k}}{k!}\Big)^{p}\exp(-a\,t^{k}/k!)=\mathbb P(\cN(t^{k}/k!)=p). 
\end{align*}
We now extend this to multidimensional marginals. Let $l\ge 2$, $0\le t_{1}\le t_{2}\le \cdots \le t_{l}$ and a sequence of non negative integers $p_{1}\le p_{2}\le \cdots \le p_{l}$. Then for $m\ge p_{l}$, we apply \eqref{gam} again to find that
\begin{align*}
&\,\mathbb P(\cN_{m}(t_{l})=p_{l}\,|\, \cN_{m}(t_{i})=p_{i}, 1\le i\le l-1)\\
=&\binom{m-p_{l-1}}{p_{l}-p_{l-1}}\Big(\mathbb P(G_{m, 1}\le t_{l}\,|\, G_{m, 1}>t_{l-1})\Big)^{p_{l}-p_{l-1}} \Big(\mathbb P(G_{m, 1}>t_{l}\,|\, G_{m, 1}>t_{l-1})\Big)^{m-p_{l}}\\
\to& \frac{1}{(p_{l}-p_{l-1})!}\Big(\frac{a\,t_{l}^{k}}{k!}-\frac{a\,t_{l-1}^{k}}{k!}\Big)^{p_{l}-p_{l-1}}\exp\Big(-\frac{a(t_{l}^{k}-t_{l-1}^{k})}{k!}\Big),
\end{align*}
which is precisely $\mathbb P(\cN(t_{l}^{k}/k!)=p_{l}\,|\, \cN(t_{i}^{k}/k!)=p_{i}, 1\le i\le l-1)$. 
Combined with an induction argument, this readily yields the distributional convergence of $(\cN_{m}(t_{i}), 1\le i\le l)$ to $(\cN_{m}(t_{i}), 1\le i\le l)$ for all $(t_{i})_{1\le i\le l}$, $l\ge 1$. Since $t\mapsto \cN_{m}(t)$ is non decreasing, we conclude with the convergence in $\mathbb D(\R_{+}, \R)$. 
\end{proof}

Recall the reduced trees $\cR^{n}_{p}$ and $\cR_{p}$. Let us take the vertices $v\in \cR^{n}_{p}$ and rank them in the increasing order of the $\eta_{v}$'s. We write the ranked sequence as $(v^{n, p}_{i})_{1\le i\le \#\cR^{n}_{p}}$ so that $\eta_{v^{n, p}_{1}}<\eta_{v^{n, p}_{2}}<\cdots<\eta_{v^{n, p}_{\#\cR^{n}_{p}}}$. %Note that the point measure $\sum_{1\le i\le M_{n, p}}\delta_{(\eta_{v^{n, p}_{i}}, v^{n, p}_{i})}$ is the restriction of $\cP_{n}$ to $\mathcal R^{n}_{p}$. 
Similarly, since $\tilde{\cP}([0, t]\times \mathcal R_{p})=\#\{(s_{i}, x_{i}): x_{i}\in \cR_{p}, s_{i}\le t\}<\infty$ for each $t>0$, we can rank the elements of $\{(s_{i}, x_{i}): x_{i}\in \cR_{p}\}$ in the increasing order of their first coordinates and write the ranked (infinite) sequence as $(\tau^{p}_{1}, \chi^{p}_{1}), (\tau^{p}_{2}, \chi^{p}_{2}), \dots$. Let us also denote
\[
\delta_{n}=\sigma^{\frac{1}{k}} n^{-\frac{1}{2k}}, \quad n\ge 1.
\]

\begin{prop}
\label{prop}
For each $p\ge 1$, as $n\to\infty$, we have for all $j\ge 1$, 
\[
\bigg(\bigg(\frac{\sigma}{\sqrt{n}}\mathcal R^{n}_{p}, v^{n, p}_{1}, \dots, v^{n, p}_{j}\bigg), \delta_{n}^{-1}\eta_{v^{n, p}_{1}}, \dots, \delta_{n}^{-1}\eta_{v^{n, p}_{j}}\bigg) \xrightarrow[n\to\infty]{(d)} \Big(\big(\cR_{p}, \chi^{p}_{1}, \dots, \chi^{p}_{j}\big), \tau^{p}_{1}, \dots, \tau^{p}_{j}  \Big),
\]
where the convergence of the first coordinates is with respect to the marked Gromov--Hausdorff topology. 
\end{prop}

\begin{proof}
Since the $\eta_{v}$'s are i.i.d, the law of $(v^{n, p}_{1},  \dots, v^{n, p}_{j})$ is that of a uniform sampling without replacement, and is further independent of $(\eta_{v^{n, p}_{i}})_{1\le i\le j}$. Combined with the convergence in \eqref{cv-R}, this implies that $(v^{n, p}_{1},  \dots, v^{n, p}_{j})$ converges in distribution to $j$ independent uniform points in $\cR_{p}$, which is precisely the distribution of $\chi^{p}_{1}, \dots, \chi^{p}_{j}$. So it remains to check the convergence of $\eta_{v^{n, p}_{i}}$. Let us define 
\[
\cN_{n, p}(t)=\sum_{v\in \cR^{n}_{p}} \mathbf 1_{\{\eta_{v}\le \delta_{n}t\}}=\max\{i: \eta_{v^{n, p}_{i}}\le \delta_{n}t\}, \quad t\ge 0. 
\]
Since each $\delta_{n}^{-1}\eta_{v}$ is distributed as an independent Gamma$(k, \frac{1}{\delta_{n}})$, applying Lemma \ref{lem: gam} with $m=\#\cR^{n}_{p}$ and $a_{m}=\delta_{n}$, we obtain from \eqref{cv-l} that $(\cN_{n, p}(t))_{t\ge 0}$ converges in distribution to 
$\cN(t^{k}/k!)_{t\ge 0}$, a Poisson process of rate $\ell(\cR_{p})$. By \eqref{defcP}, the latter has the same law as 
$(\tilde\cP([0, t]\times\cR_{p}))_{t\ge 0}$. Standard results on point processes then allow us to complete the proof. 
\end{proof}

Let $T_{n}(t)$ be the subtree of $T_{n}$ formed by the vertices connected to the root at time $t$. Note that a vertex $v\in T_{n}(t)$ if and only if none of its ancestors nor $v$ itself has been removed by time $t$. Let us denote $\mu_{n}(t)=\mu_{n}(T_{n}(t))$. Recall that $\tilde\cT_{t}$ is the subtree of $\cT$ connected to the root at time $t$ from the cutting process $\tilde{\cP}$. 
Proposition \ref{prop} implies the following
\begin{lem}
As $n\to\infty$, jointly with the convergence in \eqref{cv-R}, we have $(\mu_{n}(\delta_{n}t))_{t\ge 0}$ converging to $(\mu(\tilde\cT_{t}))_{t\ge 0}$ in distribution with respect to the Skorokhod topology on $\mathbb D(\R_{+}, \R)$. 
\end{lem}

\begin{proof}
The arguments are similar to the ones in Section 2.3, \cite{AP98}, so we'll only sketch the proof. Recall that $(V^{n}_{i})_{i\ge 1}$ (resp.~$(V_{i})_{i\ge 1}$) is a sequence of i.i.d.~uniform vertices of $T_{n}$ (resp.~i.i.d.~points of $\cT$ with law $\mu$). By Law of Large Numbers, we have for each $t>0$, 
\[
\frac{1}{j}\sum_{i=1}^{j}\mathbf 1_{\{V^{n}_{i}\in T_{n}(t)\}}\xrightarrow[j\to\infty]{\text{a.s.}} \mu_{n}(t) \quad \text{and} \quad \frac{1}{j}\sum_{i=1}^{j}\mathbf 1_{\{V_{i}\in \tilde\cT_{t}\}}\xrightarrow[j\to\infty]{\text{a.s.}} \mu(\tilde\cT_{t}) .
\]
On the other hand, $V^{n}_{i}\in T_{n}(t)$ if and only if the first $\eta_{v}$ for those $v$ in the path from the root to $V^{n}_{j}$ arrives after $t$. Therefore, according to Proposition \ref{prop}, for each $j\ge 1$, 
\[
\Big(\mathbf 1_{\{V^{n}_{i}\in T_{n}(\delta_{n}t)\}}, 1\le i\le j\Big) \xrightarrow[n\to\infty]{(d)} \Big(\mathbf 1_{\{V_{i}\in \cT(t)\}}, 1\le i\le j\Big) 
\]
It follows that we can find a sequence $k_{n}\to\infty$ slowly enough such that
\[
\frac{1}{k_{n}}\sum_{i=1}^{k_{n}}\mathbf 1_{\{V^{n}_{i}\in T_{n}(\delta_{n}t)\}} \xrightarrow[n\to\infty]{(d)} \mu(\tcT_{t}),
\]
jointly with \eqref{cv-R}. 
Invoking Law of Large Numbers again, we deduce that $\mu_{n}(\delta_{n}t)\to \mu(\tcT_{t})$ in distribution, jointly with \eqref{cv-R}. These arguments can also be adapted to prove the convergence of the multidimentional marginals. The functional convergence then follows thanks to monotonicity. 
\end{proof}

By the Skorokhod representation, we can assume from now on that jointly with \eqref{cv-R}, we have
\begin{equation}
\label{cv-mass}
\big(\mu_{n}(\delta_{n}t)\big)_{t\ge 0} \xrightarrow{n\to\infty} \big(\mu(\tilde\cT_{t})\big)_{t\ge 0} \quad \text{a.s.~in } \mathbb D(\R_{+}, \R).
\end{equation}

\subsection{Records and numbers of cuts}

Recall the Poisson process $N_{v}$ associated to each vertex $v\in T_{n}$. Let us write $\eta_{v, r}=\inf\{t: N_{v}(t)=r\}$ for the $r$-th jump of $N_{v}$; in particular, $\eta_{v, k}=\eta_{v}$. 
For $r=1, \cdots, k$, we say $v$ is a {\it $r$-record} if $v$ is still connected to the root at time $\eta_{v, r}$. Denote by $X_{k, r}(T_{n})$ the total number of $r$-records in $T_{n}$. Clearly, $X_{k}(T_{n})=\sum_{1\le r\le k}X_{k, r}(T_{n})$. On the other hand, as pointed out in Lemma 6 of \cite{CHDS19}, we have 
\begin{equation}
\label{record}
n^{-1+\frac{1}{2k}}\big(X_{k}(T_{n})-X_{k, 1}(T_{n})\big) \xrightarrow{n\to\infty} 0 \quad \text{in probability}, 
\end{equation}
so that we only need to look for the scaling limit of $X_{k, 1}(T_{n})$. To that end, let us introduce $a_{n}(t)=\#\{v\in T_{n}(t): N_{v}(t)=0\}$. Standard tools from stochastic analysis yield the following
\begin{lem}
\label{mart}
For all $n\ge 1$, we have
\[
\mathbb E\Big[\Big(X_{k, 1}(T_{n})-\int_{0}^{\infty} a_{n}(t)dt \Big)^{2}\Big]=\mathbb E\Big[\int_{0}^{\infty}a_{n}(t)dt\Big]. 
\]
\end{lem}

\begin{proof}
For $t>0$, let us denote 
\[
\mathcal X_{n}(t)=\sum_{v\in T_{n}}\mathbf 1_{\{\eta_{v, 1}\le t\}}\mathbf 1_{\{v \text{ is a 1-record}\}}, 
\]
the number of $1$-records which have occurred by time $t$. Clearly, $\mathcal X_{n}(\infty)=X_{k, 1}(T_{n})$. 
Note that $\eta_{v, 1}$ is distributed as an exponential variable with mean $1$. It is then classic that 
\[
M_{t}=\mathcal X_{n}(t)-\int_{0}^{t}a_{n}(s)ds, \quad t\ge 0,
\]
is a martingale which further satisfies that $\mathbb E[M_{t}^{2}]=\mathbb E[\int_{0}^{t}a_{n}(s)ds]$. In the terminology of point processes, this is saying that $(\int_{0}^{t}a_{n}(s)ds)_{t\ge 0}$ is the compensator of $(\mathcal X_{n}(t))_{t\ge 0}$. On the other hand, for each fixed $n$, one can easily convince oneself that $\mathbb E[\int_{0}^{\infty}a_{n}(s)ds]<\infty$. Therefore, $(M_{t})_{t\ge 0}$ is also bounded in $L^{2}$. Taking $t\to\infty$ yields the desired result. 
\end{proof}

\begin{lem}
\label{a-cv}
For each $t>0$, $\frac1n a_{n}(\delta_{n}t)-\mu_{n}(\delta_{n}t) \to 0$ in $L^{1}$. 
\end{lem}

\begin{proof}
Conditional on $\mu_{n}(t)$, $a_{n}(t)$ is distributed as Binomial$(n\,\mu_{n}(t), e^{-t})$. Hence,
 \[
 \mathbb E\Big[\Big|\frac1n a_{n}(\delta_{n}t)-\mu_{n}(\delta_{n}t)\Big|\Big]=\mathbb E[\mu_{n}(\delta_{n}t)](1-e^{-\delta_{n}t})\le \delta_{n}t\to 0,
 \]
 as $n\to\infty$. 
\end{proof}

\begin{lem}
\label{lem: bd}
We have 
\[
\lim_{t\to\infty}\limsup_{n\ge 1}\mathbb E\Big[\int_{t}^{\infty}\mu_{n}(\delta_{n}s)ds\Big]=0.
\]
\end{lem}

\begin{proof}
The first part of the proof is identical to that of Lemma 3 in \cite{BM13}. We include it here for the sake of completeness. 
Let $p(t)=\mathbb P(\eta_{v}>t)$ be the probability that $v$ is not removed at time $t$. We note that  $v\in T_{n}(t)$ if and only if $N_{w}>t$, for every vertex $w$ in the path from the root to $v$. Letting $\Ht(v)$ be the number of vertices in that path, we can write
\begin{align}
\mathbb E[n\,\mu_{n}(t)]= \mathbb E_{n}\bigg[\sum_{v\in T_{n}}\mathbb P(v\in T_{n}(t)\,|\,T_{n})\bigg] = \mathbb E\bigg[\sum_{v\in T_{n}} p(t)^{\Ht(v)}\bigg]=\sum_{m\ge 1}p(t)^{m} \mathbb E[Z_{m}(T_{n})],
\end{align}
where $Z_{m}(T_{n})=\#\{v\in T_{n}: \Ht(v)=m\}$. Now according to Theorem 1.13 in \cite{Ja06}, there exists some constant $C\in (0, \infty)$ which only depends on the offspring distribution $\xi$ such that  $\mathbb E[Z_{m}(T_{n})]\le C m$ for all $n$ and $m$. It follows that
\[
n\,\mathbb E[\mu_{n}(t)]\le C\sum_{m\ge 1} m\, p(t)^{m} =  \frac{Cp(t)}{(1-p(t))^{2}}.
\]
On the other hand, since $\eta_{v}$ has the same distribution as the sum of $k$ independent exponential variables of mean $1$, we deduce the bound $p(t)\le k \exp(-t/k)$. For small values of $t$, we will use instead:
\[
1-p(t)=\mathbb P(\eta_{v}\le t)=\int_{0}^{t} \frac{s^{k-1}}{(k-1)!}e^{-s}ds\ge e^{-t}\int_{0}^{t}\frac{s^{k-1}}{(k-1)!}ds= \frac{t^{k}}{k!}e^{-t}, \quad t\ge 0. 
\]
Let $t_{0}$ be such that $k\exp(-t_{0}/k)<1$. Applying the previous bounds, we find that for $n$ large enough, 
\begin{align*}
\mathbb E\Big[\int_{t}^{\infty}\mu_{n}(\delta_{n}s)ds\Big]&\le \frac{C}{n}\int_{t}^{\infty} \frac{p(\delta_{n}s)}{(1-p(\delta_{n}s))^{2}}\\
&\le \frac{C}{n}\int_{t}^{t_{0}/\delta_{n}} \frac{ds}{e^{-2\delta_{n}s}(\delta_{n}s)^{2k}/(k!)^{2}} + \frac{C}{n} \int_{t_{0}/\delta_{n}}^{\infty} \frac{ke^{-\delta_{n}s/k}}{(1-ke^{-\delta_{n}s/k})^{2}}ds \\
&\le \frac{C(k!)^{2}e^{2t_{0}}}{n\,\delta_{n}^{2k}}\,t^{-2k+1}+ \frac{Ck^{2}}{n\,\delta_{n}}\frac{e^{-t_{0}/k}}{1-ke^{-t_{0}/k}}\,, 
\end{align*}
where we have used a change of variable $u=ke^{-\delta_{n}t/k}$ to compute the integral over $[t_{0}/\delta_{n}, \infty)$. 
Since $n\, \delta_{n}^{2k} =\sigma^{2}$ and  $n\delta_{n}\to \infty$, the conclusion follows. 
\end{proof}

\begin{prop}
\label{prop'}
As $n\to\infty$, we have the joint convergence
\begin{equation}
\label{tt}
\Big(\frac{\sigma}{\sqrt{n}}T_{n}, \, \frac{1}{n\delta_{n}}X_{k}(T_{n})\Big) \xrightarrow{(d)} \big(\cT, X_{k}(\cT)\big), 
\end{equation}
where the convergence of the first coordinate is in the Gromov--Hausdorff sense. 
\end{prop}

\begin{proof}
We first note that $\mathbb E[X_{k}(\cT)]<\infty$ as a consequence of \eqref{bd-X} and the fact that a Rayleigh distribution has finite mean. Together with Lemma \ref{lem: bd}, this implies that for $\epsilon>0$, we can find $t_{0}=t_{0}(\epsilon)\in (0, \infty)$ such that 
\begin{equation}
\label{T}
\mathbb E\Big[\int_{t_{0}}^{\infty} \mu(\tcT_{t})dt\Big]<\epsilon \quad \text{and} \quad \mathbb E\Big[\int_{t_{0}}^{\infty} \frac1n a_{n}(\delta_{n}t)dt\Big]\le \mathbb E\Big[\int_{t_{0}}^{\infty} \mu_{n}(\delta_{n}t)dt\Big]<\epsilon, \text{ for all } n\ge 1. 
\end{equation}
Let $m\in \N$ and take $M\in\N$ large enough such that $M2^{-m}\ge t_{0}$. Since $t\mapsto a_{n}(t)$ is non increasing, we have
\[
2^{-m}\sum_{j=1}^{M} a_{n}\Big(\frac{j\,\delta_{n}}{2^{m}}\Big) \le \int_{0}^{M/2^{m}} a_{n}(\delta_{n}t)dt =\sum_{j=1}^{M}\int_{(j-1)/2^{m}}^{j/2^{m}}a_{n}(\delta_{n}t)dt \le 2^{-m}\sum_{j=1}^{M} a_{n}\Big(\frac{(j-1)\delta_{n}}{2^{m}}\Big). 
\]
Replacing $a_{n}(\delta_{n}t)$ with $\mu(\tcT_{t})$ yields a similar bound for $\int_{0}^{M/2^{m}}\mu(\tcT_{t})dt$. Then,
\begin{align*}
\Big|\int_{0}^{M/2^{m}} \frac1n a_{n}(\delta_{n}t)dt -\int_{0}^{M/2^{m}} \mu(\tcT_{t})dt \Big|\le 2^{-m+1}+2^{-m+1}\sum_{j=1}^{M}\Big|\frac1n a_{n}\Big(\frac{j\,\delta_{n}}{2^{m}}\Big)-\mu\Big(\tcT_{\frac{j}{2^{m}}}\Big)\Big|\\
\le 2^{-m+1}+2^{-m+1}\sum_{j=1}^{M}\Big\{\Big|\frac1n a_{n}\Big(\frac{j\,\delta_{n}}{2^{m}}\Big)-\mu_{n}\Big(\frac{j\,\delta_{n}}{2^{m}}\Big)\Big|+\Big|\mu_{n}\Big(\frac{j\,\delta_{n}}{2^{m}}\Big)-\mu\Big(\tcT_{\frac{j}{2^{m}}}\Big)\Big|\Big\}.
\end{align*}
As a consequence of Lemma \ref{a-cv} and \eqref{cv-mass}, we obtain  
\begin{equation}
\label{tem}
\mathbb P\Big(\limsup_{n\to\infty} \Big|\int_{0}^{M/2^{m}} \frac1n a_{n}(\delta_{n}t)dt -\int_{0}^{M/2^{m}} \mu(\tcT_{t})dt \Big|>2^{-m+1}\Big)\to 0, \quad \text{ as } n\to\infty,
\end{equation}
jointly with the convergences in \eqref{cv-R}. 
On the other hand, Lemma \ref{mart} and a change of variable yield 
\[
\mathbb E\Big[\Big(\frac{X_{k, 1}(T_{n})}{n\delta_{n}}-\int_{0}^{\infty} \frac{1}{n}a_{n}(\delta_{n}t)dt \Big)^{2}\Big]=\frac{1}{n\delta_{n}}\mathbb E\Big[\int_{0}^{\infty}a_{n}(t)dt\Big]\to 0, 
\]
as $n\to\infty$. Combining this with \eqref{tem}, \eqref{T} and then \eqref{record}, we obtain 
\begin{equation}
\label{temp}
\frac{1}{n\delta_{n}}X_{k}(T_{n}) \xrightarrow{n\to\infty} X_{k}(\cT) \quad \text{ in probability}, 
\end{equation}
jointly with the convergences in \eqref{cv-R}. Combined with \eqref{cv-tree}, this shows the convergence of both marginals in \eqref{tt}. To get to the joint convergence, it suffices to note that the law of $(\cT, X_{k}(\cT))$ is the unique limit point of those on the left-hand side, which follows from the joint convergence in \eqref{temp} and the fact that the family $(\cR_{p})_{p\ge 1}$ uniquely determines the law of $(\cT, d, \mu)$. 
\end{proof}

\begin{proof}[{\bf Proof 2 of Theorem \ref{thm1}}]
This follows by comparing the convergence in Proposition \ref{prop'} with \eqref{cv: k-cut}. 
\end{proof}

{\small
\setlength{\bibsep}{.2em}

}

\end{document}